\documentclass[12pt,a4paper]{amsart}
 \pdfoutput=1
 
 \usepackage{amsmath,amssymb,amsthm}
 \usepackage{bm}  
 \usepackage{mathtools}
 \usepackage{thmtools}
 \usepackage{float} 
 
 \usepackage{xcolor} 	
 \usepackage{hyperref}
 \hypersetup{
 	colorlinks,
     linkcolor={red!60!black},
     citecolor={green!60!black},
     urlcolor={blue!60!black},
 }
 \usepackage[maxbibnames=99]{biblatex}
 \usepackage[utf8]{inputenc}
  \usepackage{csquotes}
 \usepackage[T1]{fontenc}
 \usepackage{lmodern}
 \usepackage[babel]{microtype}
 \usepackage[english]{babel}
 
 \usepackage{geometry}
 \usepackage{enumitem}

\DeclareMathAlphabet{\mymathbb}{U}{BOONDOX-ds}{m}{n}

\theoremstyle{plain}
\newtheorem{theorem}{Theorem}[section]
\newtheorem{fact}[theorem]{Fact}

\newtheorem{claim}[theorem]{Claim}
\newtheorem{corollary}[theorem]{Corollary}
\newtheorem{lemma}[theorem]{Lemma}

\newtheorem{observation}[theorem]{Observation}

\newtheorem{conjecture}[theorem]{Conjecture}

\newtheorem*{theorem*}{Theorem}
\newtheorem*{corollary*}{Corollary}

\theoremstyle{definition}

\newtheorem{definition}[theorem]{Definition}

\title{Hindrance from a wasteful common independent set}

\author{Attila Jo\'{o}}
\thanks{Funded by the Deutsche Forschungsgemeinschaft (DFG, German Research Foundation) - Grant No. 513023562}
\address{Attila Jo\'{o},
Department of Mathematics, University of Hamburg, Bundesstra{\ss}e 55 (Geomatikum), 20146 Hamburg, Germany}
\email{attila.joo@uni-hamburg.de}

\keywords{}
\subjclass[2020]{Primary 05B35. Secondary 03E05, 05C63.} 
\begin{document}
\begin{abstract}
For (potentially infinite) matroids $ M $ and $ N $, an $ (M,N) $-hindrance 
is a set $ H$ that is independent but not spanning in $ N.\mathsf{span}_M(H) $. This concept was 
introduced by Aharoni and Ziv in the very first paper investigating Nash-Williams' Matroid Intersection Conjecture.  They 
proved that the conjecture is equivalent to the statement that the non-existence of hindrances implies the existence of an $ M 
$-independent spanning set of $ N $. 

In this paper we present a breakthrough towards the Matroid Intersection Conjecture. Namely, we found a matroidal 
generalization of the `popular 
vertex' approach applied in the proof of the infinite version of König's theorem.  The main result of this paper is an 
application of this new approach to show that 
if $ M $ and $ 
N $ admit a common independent set $ I $ that is ``wasteful'' in the sense 
that $ r(M/I)<r(N/I) $, then there exists an $ (M,N) $-hindrance. 
\end{abstract}
\maketitle

\section{Introduction}
The infinite version of König's Duality Theorem \cite{aharoni1984konig} states that every (finite or infinite) bipartite 
graph  $ G=(S,T,E) $ admits a matching $ F\subseteq E $ such that one can obtain a vertex-cover of $ G $ by choosing precisely 
one endpoint of
each $ e \in F $. This is a special case of the following conjecture  (known as the Matroid Intersection Conjecture) of 
Nash-Williams 
\cite[Conjecure 1.2]{aharoni1998intersection}:  If $ M $ and $ N 
$ are finitary 
matroids\footnote{For matroidal terminology see Subsection \ref{subsec: matroids}.} on a common edge set $ E $, then 
there exists a common independent set $ I $ and a partition $ E=E_M\cup E_N $ such that $ I\cap E_M $ spans $ E_M $ in $ M $ and 
$ 
I\cap E_N $ spans 
$ E_N $ in $ N$. 

In the proof of the infinite version of König's Duality Theorem, the starting point is an asymmetric reformulation. A bipartite 
graph $ 
G=(S,T,E) $ is called hindered if there is an $ X\subseteq T $ for which the neighbourhood $ N_G(X) $ of $ X $ can be matched into a proper 
subset of $ X $.  
The 
reformulation in question says that if $ G=(S,T,E) $ is unhindered, then it admits a 
matching that covers $ T $. A similar  asymmetric reformulation, introduced in \cite{aharoni1998intersection}, turned out to be useful in the 
context of the Matroid Intersection Conjecture as 
well. The ordered pair $ (M,N) $ of matroids is called hindered if there is an  $ H\subseteq E $ which is independent but not 
spanning in $ 
N.\mathsf{span}_M(H) $. 
The following statement is known to be equivalent to the Matroid Intersection Conjecture and it is usually the preferred angle 
to approach it: 
\begin{conjecture}[{\cite[Conjecture 3.3]{aharoni1998intersection}}]
If $ (M,N) $ is unhindered, then there exists an $ M $-independent set that is spanning in $ N $.
\end{conjecture}

A key ingredient of Aharoni's proof of König's Duality Theorem for uncountable graphs is the concept of `popular vertices'. 
Roughly speaking, a 
vertex $ v $ is popular with respect to a matching if there are ``a lot of''  augmenting paths that terminate at $ v $ and are disjoint apart from $ v $. 
The 
simplest example where the `popular vertex' method can be applied is the proof that $ G=(S,T,E) $ is hindered under the assumption that there is a
matching $ F $ that is ``wasteful''\footnote{Wasteful from the perspective of searching for a matching that covers $ T $.} in the sense that more 
vertices in $ T $ are 
uncovered by $ F $ than in $ S $. We gave a direct proof for this in 
\cite{joo2025hindranclinkage} in a 
more general setting.

After 
the Matroid Intersection 
Conjecture was settled for countable matroids in \cite{joo2021MIC}, we were searching for an adequate matroidal 
generalization of Aharoni's 
`popular 
vertex' technique. These efforts were eventually successful, and the resulting methods are explained in this paper by proving 
the matroidal generalization of creating hindrances from wasteful matchings:

\begin{theorem}\label{thm: main}
If $ M $ and $ N $ are finitary matroids on a common ground set and there exists an $ I \in \mathcal{I}(M)\cap \mathcal{I}(N) $ with $ 
r(M/I)<r(N/I) $, then $ (M,N) $ is hindered. 
\end{theorem}

The setting of  König's Duality Theorem 
corresponds to the special case of the Matroid Intersection 
Conjecture where in each matroid every circuit has size two. The main difficulty is when we can have arbitrarily large circuits
that we might need multiple edges to span an edge and hence a single augmenting path is insufficient. Instead of popular vertices (which are 
meaningless in the matroidal context for obvious reasons) we define something that we call the `popularity matroid', which serves as a key concept 
in the proofs.

The paper is organized as follows. In the following section we introduce our notation and provide a survey of the set-theoretic and matroidal 
facts we need later. Then, in Section \ref{sec: prep}, we prove some preparatory lemmas. Finally, our main result, Theorem \ref{thm: main}, is 
proved in Section \ref{sec: main proof}.

\section*{Acknowledgment} 
We would like to thank Nathan Bowler, who listened carefully to our proofs and pointed out gaps in  previous versions.
\section{Notation and Preliminaries}
\subsection{Set theory}
We denote by $ \boldsymbol{\bigcup \mathcal{F}} $ the union of the sets in $ \mathcal{F} $. The symmetric difference $ \boldsymbol{X 
\triangle 
Y} $ is defined as $ (X\setminus Y)\cup (Y\setminus X) $. A $ \triangle $\emph{-system} $ \boldsymbol{\mathcal{D}} $ is a set of sets where 
any two elements of $ \mathcal{D} $ have the same intersection $ \boldsymbol{K} $, which is called the \emph{kernel}\footnote{All the $ 
\triangle $-systems we use will have at least two elements; thus, the kernel will be unique.} of $ \mathcal{D} 
$. The sets $ C\setminus K $ for $ C\in \mathcal{D} $ are the \emph{petals} of $ \mathcal{D} $. The 
variables $ \boldsymbol{\alpha}$, $\boldsymbol{\beta} $, and $ \boldsymbol{\gamma} $  represent
ordinal numbers, while $\boldsymbol{ \kappa} $ 
and $\boldsymbol{ \lambda} $ 
denote 
cardinals. The smallest 
limit ordinal, i.e. the set of the natural numbers is denoted by $ \boldsymbol{\omega} $.  An infinite cardinal $ \kappa $ is \emph{regular} if it is 
not the union of 
less than $ \kappa $ sets each of which has cardinality less than $ \kappa $.  Let $ \kappa $ be an uncountable regular cardinal. A set $ U\subseteq 
\kappa $ is a \emph{club} of $ \kappa $ if it 
  is unbounded in $ 
 \kappa $ and closed with respect to 
 the order 
 topology (i.e. $ \sup B:=\bigcup B \in U $ for every $ B\subseteq U $ bounded in $ \kappa $).  A set $ 
 S\subseteq \kappa $ is \emph{stationary} if $ \kappa \setminus S $ does not include a club. 
\begin{fact}[{\cite[Theorem 8.3]{jech2002set}}]\label{fact: union of non-stat}
For every uncountable regular cardinal $ \kappa $, the union of less than $ \kappa $ many non-stationary subsets of $ \kappa $ is 
non-stationary.
\end{fact}
\begin{lemma}[{Fodor's lemma, \cite[Theorem 8.7]{jech2002set}}]\label{lem: Fodor's lemma}
Let $ \kappa $ be an uncountable regular cardinal, let $ S\subseteq \kappa $ be stationary and let $ f: S \rightarrow \kappa $ 
be 
a regressive function, i.e. $ f(\alpha)<\alpha $ for every $ \alpha \in S \setminus \{ 0 \} $. Then there is a stationary $ 
S'\subseteq S $ such that the restriction of $ f $ to $ S' $ is a constant function.
\end{lemma}
\begin{corollary}\label{cor: Fodor 1}
Let $ \kappa $ be an uncountable regular cardinal and for $ \alpha<\kappa $, let $ O_\alpha \subseteq \kappa \setminus (\alpha+1) 
$ be a non-stationary set. Then $ O:=\bigcup_{\alpha<\kappa}O_\alpha $ is non-stationary.
\end{corollary}
\begin{proof}
Suppose for a contradiction that $ O $ is stationary and for $ \beta \in O $, let $ f(\beta) $ be the smallest ordinal $ \alpha $ for which $ \beta \in 
O_\alpha $. Note that $ f(\beta)<\beta $ because $ O_\alpha \subseteq \kappa \setminus (\alpha+1) 
$ for each $ \alpha $. Thus by Fodor's lemma there is an $ \alpha $ such that $ f(\beta)=\alpha $ for stationarily many $ \beta 
$, which implies that $ 
O_\alpha $ is stationary, a contradiction.
\end{proof}

\subsection{Digraphs}
By a digraph $ D $ we mean a set of ordered pairs that we call the \emph{arcs} of $ D $ (while the vertex set of all the digraphs will be the edge 
set $ 
E $ 
of some matroids). For $ F\subseteq E $, we write $ \boldsymbol{\mathsf{in}_D(F)} $ and $ \boldsymbol{\mathsf{out}_D(F)} $, for the 
\emph{in-neighbours} and \emph{out-neighbours} of $ F 
$ in $ D $ respectively. This means that $ \mathsf{in}_D(F) $ consists of those $ e\in E\setminus F $ for which there exists an $ f\in F $ such that
$ ef\in D $. The set $ \mathsf{out}_D(F) $ is defined analogously. The \emph{in- and out-degree} of $ e $ in $ D $ is $ 
\left|\mathsf{in}_D(e)\right| $ and $ \left|\mathsf{out}_D(e)\right| $ respectively. By a \emph{path} in a digraph we always mean a finite 
directed path. For a path $ P $, we write $ 
\boldsymbol{\mathsf{in}(P)} $ and $ \boldsymbol{\mathsf{ter}(P)} $ for the initial and terminal vertex of $ P $ respectively. 
For a set $ 
\mathcal{P} $ of paths let $\boldsymbol{ \mathsf{in}(\mathcal{P})}:=\{ \mathsf{in}(P):\ P \in \mathcal{P} \} $ and $\boldsymbol{ 
\mathsf{ter}(\mathcal{P})}:=\{ \mathsf{ter}(P):\ P \in \mathcal{P} \} $.

\subsection{Matroids}\label{subsec: matroids}
A pair ${M=(E,\mathcal{I})}$ is a \emph{finitary matroid} 
if ${\mathcal{I} \subseteq \mathcal{P}(E)}$ satisfies  the following axioms: 
\begin{enumerate}
[label=(\Roman*)]
\item\label{item: matroid axiom1} $ \emptyset\in  \mathcal{I} $;
\item\label{item: matroid axiom2} If $I\in \mathcal{I} $ and $ J\subseteq I $, then $ J\in \mathcal{I} $;
\item\label{item: matroid axiom3} If $ I,J\in \mathcal{I} $ are finite with $ \left|I\right|<\left|J\right| $, then there exists an $  e\in J\setminus I $ 
for which
$ I\cup \{ e \}\in \mathcal{I} $;
\item\label{item: matroid axiom4} If all finite subsets of  a set $ I\subseteq E $ are in $ \mathcal{I} $, then $ I\in \mathcal{I} $.
\end{enumerate}

We refer to $ E $ as the edge set or ground set of the matroid.
The sets in~$\mathcal{I}$ are called \emph{independent} and the sets $ \mathcal{P}(E)\setminus \mathcal{I} $ are \emph{dependent}. We write 
$ \boldsymbol{\mathcal{I}(M)} $ if the matroid is not clear from the context. The  maximal independent 
sets are the \emph{bases}, while the minimal dependent sets are the \emph{circuits} of a matroid. We write $ \boldsymbol{\mathcal{B}(M)} $ 
and $ 
\boldsymbol{\mathcal{C}(M)} $ for 
 the bases and the circuits of $ M $ respectively. The bases of a finitary matroid $ M $ have the same size, which is the 
\emph{rank} of $ M $ and is denoted by $ \boldsymbol{r(M)} $. A singleton circuit is called a \emph{loop}. A set $ \mathcal{C} $ of sets 
satisfies the \emph{circuit elimination axiom} 
if 
for every 
distinct $ C_0, C_1\in \mathcal{C} $, and $ e\in C_0 \cap C_1 $ there is a $ C\in \mathcal{C} $ with $ C\subseteq (C_0\cup C_1)\setminus \{ e \} 
$. Moreover, $ \mathcal{C} $ satisfies the \emph{strong circuit elimination axiom} if for every 
 $ C_0, C_1\in \mathcal{C} $, $ e\in C_0 \cap C_1 $, and $ f\in C_0 \triangle C_1 $ there is a $ C\in \mathcal{C} $ with $f\in C\subseteq 
 (C_0\cup C_1)\setminus \{ e \} $.  A set $ \mathcal{C} $ of finite subsets of $ E $ is the set of the circuits of a matroid iff $ \emptyset \notin 
 \mathcal{C} $, the elements of $ \mathcal{C} $ are pairwise $ \subseteq $-incomparable, and $ \mathcal{C} $ satisfies the circuit elimination 
 axiom. These are called the \emph{circuit axioms} and provide an axiomatization of matroids in terms of circuits. The 
 circuit 
 elimination axiom
 implies the strong circuit elimination axiom under the assumption that the elements of $ \mathcal{C} $ are $ \subseteq 
 $-incomparable.

 For an  ${X \subseteq E}$, ${\boldsymbol{M  \!\!\upharpoonright\!\! X} 
:=(X,\mathcal{I} 
\cap \mathcal{P}(X))}$ is a matroid and it  is called the \emph{restriction} of~$M$ to~$X$. Let  $ \boldsymbol{r_M(X)}:=r(M 
\!\!\upharpoonright\!\! X) $.
We write ${\boldsymbol{M \setminus X}}$ for $ M  \!\!\upharpoonright\!\! (E\setminus X) $. Let $ B $ be a base of $ M  \!\!\upharpoonright\!\! 
X $.
The \emph{contraction} of $ X $ in $ M $ is the matroid 
$\boldsymbol{M/X}=(E\setminus X, \mathcal{I}')$ where $ I \in \mathcal{I}' $ iff $ B\cup I \in \mathcal{I} $. This is well defined, i.e. $ 
\mathcal{I}' $ does not depend on the choice of $ B $. Let $\boldsymbol{M.X}:= M/(E\setminus X) $.   
If $ I $ is  independent in $ M $  but $ I\cup \{ e \} $ is dependent for some $ e\in E\setminus I $,  then there is a unique 
circuit   $ \boldsymbol{C_M(e,I)} $ of $ M $ through $ e $ included in $I\cup \{ e \} $ which is called the \emph{fundamental circuit} of $ e $ on 
$ I $.  We say~${X 
\subseteq E}$ \emph{spans}~${e \in E}$ in matroid~$M$ if either~${e \in X}$ or there exists a circuit~${C 
\ni e}$ with~${C\setminus \{ e \} \subseteq X}$. 
By letting $\boldsymbol{\mathsf{span}_{M}(X)}$ be the set of edges spanned by~$X$ in~$M$, we obtain a finitary closure 
operator
$ \mathsf{span}_{M}: \mathcal{P}(E)\rightarrow \mathcal{P}(E) $. 
An ${S \subseteq E}$ is \emph{spanning} in~$M$ if~${\mathsf{span}_{M}(S) = E}$.  For $ I\in \mathcal{I}(M) $ and 
 $ X\subseteq 
\mathsf{span}_M(I) $, let $ \boldsymbol{G_{M}(X,I)} $ be the 
\emph{generator} of $ X $ 
w.r.t. $ I $, 
i.e. the unique 
minimal $ I'\subseteq I $ that $ M $-spans $ X $.  If $ X $ is a singleton $ \{ e \} $, we leave the brackets. Note that $ 
G_{M}(e,I) $ is $ \{ e \} $ iff $ e\in I $, and $ C_M(e,I)\setminus \{ e \} $ otherwise. In general $ G_{M}(X,I)=(X \cap I)\cup \bigcup_{e\in 
X\setminus I}C_M(e,I)\setminus \{ e \}  $. 

\begin{claim}\label{clm: arc from circle}
If $ C\in \mathcal{C}(M) $ and $ B\in \mathcal{B}(M) $, then for every $ f\in C\cap B $ there is an $ e\in 
C\setminus B $ such that $ f\in C_M(e,B) $.
\end{claim}
\begin{proof}
Suppose for a contradiction that for $ f\in C\cap B $ there is no suitable $ e $. Then $ G_M(C\setminus \{ f \}, B) $ does not contain $ f $, yet it 
spans 
$ f $, which is a contradiction because it means that  $ G_M(C, B)\subseteq B \in \mathcal{I}(M)$ 
includes a circuit through $ f $.
\end{proof}

An $ S\subseteq E $ is called a \emph{scrawl} of $ M $ if $ S $ is the union of $ M $-circuits.
\begin{lemma}\label{lem: scrawl}
Let  $ \mathcal{F} $ be a set of finite nonempty subsets of $ E $ that satisfies the strong circuit elimination axiom. Then the set $ \mathcal{C} 
$ of the minimal elements of $ 
\mathcal{F} $ is the set of the circuits of a finitary matroid $ M $ on $ E $ and each $ F\in \mathcal{F} $ is a scrawl of $ M $.
\end{lemma}
\begin{proof}
It is straightforward to check that $ \mathcal{C} $ satisfies the circuit axioms. To prove the second part, let $ F\in \mathcal{F} $ and $ f\in F $ 
be arbitrary. We need to 
find a $ C\in \mathcal{C} $ with $ f\in C \subseteq F $. Let $ C\in \mathcal{F} $ be  minimal 
under the condition that $ f\in C \subseteq F $. Suppose for a contradiction that $ C\notin \mathcal{C} $. Then there is a $ 
G\in \mathcal{F} $ with $ G\subseteq 
C\setminus \{ f \} $ . Let $ e\in G $ be arbitrary. By applying the strong circuit elimination axiom with $ C $, $ G $, $ e $, 
and $ f $, we 
obtain a $ H\in \mathcal{F} $ with $ f\in H \subseteq C\setminus \{  e \}  $. This contradicts the minimality of $ C $. 
\end{proof}

Let  $ \boldsymbol{(M,N)} $ be an ordered pair of finitary matroids on a common edge set $ E $.  An $ (M,N) $\emph{-hindrance} 
(\cite[Definition 3.2]{aharoni1998intersection}) is an $ H\subseteq E $ with  $H\in \mathcal{I}(N.\mathsf{span}_M(H)) \setminus 
\mathcal{B}(N.\mathsf{span}_M(H)) $. 
We call $ (M,N) $ 
\emph{hindered} if there exists an $ (M,N) $-hindrance.

\begin{observation}\label{obs:  hindrance minor}
For every $ A\subseteq E $, every $ (M\setminus A, N/A) $-hindrance is an $ (M,N) $-hindrance as well.
\end{observation}
\begin{proof}
Let $ H $ be an $ (M\setminus A, N/A) $-hindrance. This means that  $ H $ is independent but not spanning 
in $(N/A).\mathsf{span}_{M\setminus A}(H)=N.\mathsf{span}_{M\setminus A}(H) $.  Since $ 
N.\mathsf{span}_{M\setminus A}(H) $ is a contraction minor of $ 
N.\mathsf{span}_M(H) $, the assumption that $ H $ is independent but not spanning in $ N.\mathsf{span}_{M\setminus A}(H)  $ implies the 
same 
for $ N.\mathsf{span}_M(H) $.
\end{proof}

The following lemma is a well-known consequence of the classical `augmenting path'-method. Informally, it says that for a given common 
independent set $ I $, we either find a slightly ``better'' common independent set  (by changing along an augmenting path in a 
certain auxiliary 
digraph)  or there is a bipartition of $ E $ that together with $ I $ satisfies 
the conditions of
Matroid Intersection Conjecture. It is discussed in detail, for example, in 
\cite[Section 3]{joo2021MIC}.
\begin{lemma}[Augmenting path lemma]\label{lem: no aug hindrance}
For every $ I\in \mathcal{I}(M)\cap \mathcal{I}(N)$ there is either a $ J\in \mathcal{I}(M)\cap \mathcal{I}(N) $ with $ r_{M/I}(J\setminus 
I)=r_{N/I}(J\setminus I)=1 $ or a partition $ E=E_M\cup E_N $ such that $ I\cap E_M $ spans $ E_M $ in $ M $ and $ I\cap E_N $ spans 
$ E_N $ in $ N$. If the second case occurs, and $ I $ does not span $ N $, then $ I \cap E_M $ constitutes an $ (M,N) $-hindrance
\end{lemma}

\section{Preparations}\label{sec: prep}
Let an ordered pair $ (M,N) $ of finitary matroids on $ E $ be fixed in this section.
\subsection{The changes of a fundamental circuit}\label{subsec: fund circ}
\begin{lemma}
Suppose that $ C_0, C_1 \in \mathcal{C}(M) $, $ f\in C_0 \cap C_1 $, $ e_i\in C_i\setminus C_{1-i} $ for $ i\in \{ 0,1 \} $, 
and $ (C_0\cup C_1) 
\setminus \{ e_0, e_1 \}\in \mathcal{I}(M)$. Then there is a unique circuit $ C\subseteq (C_0\cup C_1)\setminus \{ f \} $. 
Furthermore, 
$ C_0 \triangle C_1 \subseteq C $ holds for this $ C $.
\end{lemma}

\begin{proof}
The existence of a circuit $ C\subseteq (C_0\cup C_1)\setminus \{ f \} $ follows by circuit elimination. First we show that 
any such  $ C $ includes $ \{ e_0, e_1 \} $. Since $ C \cap \{ e_0, e_1 \}=\emptyset  $ is impossible due to the assumption $ I:=(C_0\cup C_1) 
\setminus \{ e_0, e_1 \}\in \mathcal{I}(M)$, we may assume by symmetry that $ e_0 \in C $. Suppose for a contradiction 
that $ e_1\notin C $. Then $ C $ is the fundamental circuit of $ e_0 $ on $ I $ and hence $ C=C_0 $ because  $ C_0 $ is as well. 
This is a contradiction because $ f \in C_0 \setminus C $. Thus every circuit $ C\subseteq (C_0\cup C_1)\setminus \{ f \} $ 
includes $ \{ e_0, e_1 \} $ and hence it is the fundamental circuit of $ e_0 $ on $ I\setminus \{ f \}\cup \{ e_1 \} $. 
Therefore $ C $ is unique. 

It remains to show $ C_0 \triangle C_1 \subseteq C $. Suppose that $ e_i'\in C_i\setminus C_{1-i} 
$ is 
distinct from $ e_i $. Then $I':= I \cup \{ e_i \} \setminus \{ e_i' \} $ is independent, thus the conditions of the lemma are intact if 
we replace $ I $ by $ I' $ and  $ e_i $ by $ e_i' $. Since $ C_0, C_1 $, and $ f $ are unchanged, so is the unique $ C $ included in $ (C_0\cup 
C_1)\setminus \{ f \} $ and we have seen that $ C $
must include $ \{ e_i', e_{1-i} \} $.
\end{proof}

The previous lemma will be used to describe how fundamental circuits change after a base exchange. Let us restate it in the following form:
\begin{corollary}\label{cor: fundamental change}
Suppose that $ B\in \mathcal{B}(M) $, $ e_0, e_1\in E\setminus B $  are distinct, and $ f\in C_M(e_0, B) \cap 
C_M(e_1, B) $. Let $ B':= B\cup \{ 
e_0 \}\setminus \{ f \} $. Then \[ C_M(e_0, B) \triangle C_M(e_1, B)  \subseteq C_M(e_1, B')\subseteq   C_M(e_0, B) \cup 
C_M(e_1, B)\setminus \{ f \}. \]
\end{corollary}

The next lemma will be useful in the proof of our key lemma (Lemma \ref{lem: key}).
\begin{lemma}\label{lem: circuit eliminate F}
Assume that  $ B\in \mathcal{B}(M) $, $ e_0,\dots, e_{n}\in E \setminus B $ are pairwise distinct and $ F \subseteq 
B $ with $ \left|F\right|\leq n $. Then there is an $ i^{*}\leq n $ and a 
$ C\in \mathcal{C}(M) $ with  $e_{i^{*}} \in C\subseteq 
\bigcup_{i^{*}\leq i \leq n}C_M(e_i, B)\setminus F $.
\end{lemma}
\begin{proof}
We apply induction on $ n $. For $ n=0 $, we  have $ F=\emptyset  $  in which case $ i^{*}=0 $ and $ C=C_M(e_0, B) $ are 
as desired.  
Now we show the statement for $ n+1 $ assuming we know it for $ n $. We may assume that $ F\cap C_M(e_{n+1}, B) \neq 
\emptyset  $ since 
otherwise $ i^{*}=n+1 $ and $ C=C_M(e_{n+1}, B) $ are suitable. Let $ f\in F\cap C_M(e_{n+1}, B) $ and consider the 
base $ B':=B \cup \{ e_{n+1} \}\setminus \{ f \} $. Obviously, $ f\notin C_M(e_i, B') $ because $ f\notin B' $, moreover, Corollary \ref{cor: 
fundamental change} 
ensures that for every $ i\leq n $, $ C_M(e_i, 
B')\setminus C_M(e_i, B) \subseteq C_M(e_{n+1}, B) $. But then we are done by applying the 
induction hypotheses with $ B' $, $ 
e_0,\dots, e_{n}\in E \setminus B' $ and $ F\setminus \{ f \}\subseteq B' $.
\end{proof}

\subsection{A key lemma}\label{subsec: key lem}
\begin{lemma}\label{lem: key}
Assume that  $ B\in \mathcal{B}(M) $ and $ H\subseteq E\setminus B $ where  $ \kappa:= 
\left|H\right| $ is an 
uncountable 
regular 
cardinal. Suppose that there is  a $ B'\subseteq B $ with $ \left|B'\right|<\kappa $ 
such that $ C_M(e,B) \cap B'\neq \emptyset $ for each $ e\in H $.
Then there is an $ e^{*}\in H $ and a $ \kappa $-sized $ \triangle $-system $ \mathcal{D} $ of circuits with kernel $ K $ 
such that $ 
\bigcup \mathcal{D} \subseteq\bigcup_{e\in H}C_M(e,B) $ and  $ 
e^{*}\in K \subseteq  C_M(e^{*},B)\setminus B'$.
\end{lemma}
\begin{proof}
Let $ \{ e_\alpha:\ \alpha<\kappa \} $ be an enumeration of $ H $, $ C_\alpha:=C_M(e_\alpha, B) $ and let $ 
C_\alpha^{-}:=C_\alpha\setminus C_{<\alpha} $ where $ C_{<\alpha}:=\bigcup_{\beta<\alpha}C_\beta $. By 
deleting the irrelevant edges in $ E $, we may assume without loss of generality that $ 
E=\bigcup_{\alpha<\kappa}C_\alpha $, and thus  $ H=E\setminus B $. Note that $ E=\bigcup_{\alpha<\kappa}C_\alpha^{-} $ is a 
partition.  Since $ \left|B'\right|<\kappa $ by
assumption, and $ \kappa $ is regular,   $\alpha^{*}:=\sup\{ \alpha+1:\  B'\cap C_\alpha^{-}\neq \emptyset \}<\kappa $. Note 
that for each $ \alpha\in 
\kappa \setminus \alpha^{*} $, the circuit $ C_\alpha $ meets $ C_{<\alpha^{*}} $ because $  
C_\alpha \cap B'\neq \emptyset $. Suppose for a contradiction 
that the statement is false. Then in particular there is no  $ \alpha $ with $ \alpha^{*}\leq \alpha<\kappa $ that $ e_\alpha $ satisfies the 
requirements for $ e^{*} $. This implies that  there is no $ K$   with
$e_\alpha \in K \subseteq C_\alpha^{-}  $  such that $ K $ admits $ \kappa $ many pairwise disjoint extensions to a circuit. It follows that we can 
delete less than $ 
\kappa 
$ 
edges to destroy all possible extension. More precisely, for 
every $ \alpha\in \kappa \setminus \alpha^{*} $ we can pick a set $ O_\alpha \subseteq \kappa \setminus (\alpha+1) $ with $ 
\left|O_\alpha 
\right|<\kappa$ such that every circuit through $ e_\alpha $  meets 
 $ C_{<\alpha}\cup  \bigcup_{\beta \in O_\alpha}C_{\beta}^{-} $. Since $ O_\alpha $ is in particular 
non-stationary, Corollary \ref{cor: Fodor 1} ensures that so is $ O:=\bigcup_{\alpha^{*}\leq\alpha<\kappa}O_\alpha $.  Then 
$ S:=\kappa 
\setminus 
(\alpha^{*}\cup O) $ is stationary and $ C_\alpha \supsetneq C_\alpha^{-} $ for $ \alpha \in S $ by the definition of $ 
\alpha^{*} $ and by $ C_\alpha \cap B'\neq \emptyset $. 
\begin{claim}
There is a stationary set $ S'\subseteq S $ and a set $ F $ such that $ C_\alpha 
\setminus C_\alpha^{-}=F $ for every $ \alpha \in S' $
\end{claim}
\begin{proof}
The proof is a routine application of Fodor's lemma and the $ \kappa $-additivity of the non-stationary ideal. First of all, by Fact \ref{fact: union of 
non-stat} there is a 
stationary $ S_0\subseteq S $ and $ n\in \omega \setminus \{ 0 \} $ such that $ \left|C_\alpha 
\setminus C_\alpha^{-}\right|=n $ for every $ \alpha\in S_0 $. Let $ m\leq n $ be the largest number for which there is a set $ 
F $ of size $ m $ and a 
stationary set $ S_1\subseteq S_0 $ such that for every  $ \alpha \in S_1 $,  $F\subseteq  C_\alpha  \setminus C_\alpha^{-} $. 
Suppose for a contradiction that $ m<n $. The map $ S_1 \ni \alpha \mapsto \min \{ \beta<\kappa:\ 
(C_\alpha  \setminus (C_\alpha^{-}\cup F))\cap C_\beta^{-}\neq \emptyset  \}  $ is regressive; thus, by Fodor's lemma 
\ref{lem: Fodor's lemma} there is a stationary $ S_2\subseteq 
S_1 $ where it is constant, $ \beta_0 $ say. Then for every $ \alpha \in S_2 $, $ C_\alpha  \setminus (C_\alpha^{-}\cup F) $ 
meets $ C_{\beta_0}^{-} $. But then there is a stationary $ S_3\subseteq S_2 $ and $ e\in C_{\beta_0}^{-} $ such that for 
every $ \alpha \in S_3 $, $e\in C_\alpha  \setminus (C_\alpha^{-}\cup F) $. Since $ \left|F\cup \{ e \}\right|=m+1 $, this 
contradicts the maximality of $ m $. Thus $ S':=S_1 $ and $ F $  are as desired.
\end{proof}

Let $ S' $, $ F $ as in the claim above, $ n:=\left|F\right| $ and let  $ \alpha_0< \alpha_1< 
\dots <\alpha_n,  $ be the first $ n+1 $  ordinals in $ S' $. We pick $ i^{*}\leq n $  and  circuit $ C $ 
by  
applying Lemma \ref{lem: circuit eliminate F} with 
$ B $, $ e_0=e_{\alpha_0},\dots, e_n=e_{\alpha_n} $ and $ F $. Then  $e_{\alpha_{i^{*}}}\in  C $ and
 $C \cap (C_{<\alpha_{i^{*}}}^{-}\cup\bigcup_{\alpha \in O_{\alpha_{i^{*}}}}C_{\alpha}^{-})=\emptyset  $, which 
 contradicts the 
 choice of $ O_{\alpha_{i^{*}}} $.
\end{proof}

\subsection{A new auxiliary digraph}\label{subs: aux digraph}
The classical auxiliary digraph corresponding to a common independent set is insufficient for our purpose. Here we introduce a new 
one with more arcs. Recall that the vertex set of all digraphs we define is $ E $ and,  to simplify the notation, we identify the digraphs 
with their arc sets.
For $ B\in \mathcal{B}(M) $, let $ \boldsymbol{D_M(B)} $  be the digraph in which $ ef\in D_M(B) $ iff $ e\in 
E\setminus B_M $   and $ f\in C_M(e,B)\setminus \{ e \} $. For a digraph $ D $, let $ \boldsymbol{D^{-1}} $ that we obtain 
by reversing all the 
arcs of $ D $.   Finally, for $ (B_M,B_N)\in \mathcal{B}(M)\times \mathcal{B}(N) $,
we set $ \boldsymbol{D_{(M,N)}(B_M,B_N):}=D_M(B_M)\cup D_N^{-1}(B_N) $. We will omit the subscripts whenever 
the matroids are clear 
from 
the context, and we will write $ \boldsymbol{b} $ for $ (B_M,B_N) $. 

Note 
that for each  $ef\in  D(b)$ there are one or two circuits corresponding to $ ef $ depending on weather $ ef\in D(B_M)\triangle 
D^{-1}(B_N) $ or $ ef\in D(B_M)\cap D^{-1}(B_N) $. These are the \emph{escorting circuits} of $ ef $ w.r.t. $ (M,N) $ 
and $ b $.   Each arc  
$ef\in  D(B_M) $ indicates 
a 
base exchange 
in $ M $ w.r.t. $ B_M $, namely the removal of 
$ f $ and the addition of $ e $. Similarly,  if $ ef\in D^{-1}(B_N) $, then the indicated base exchange in $ N $ 
w.r.t. $ B_N $ is the addition of $ f $ and  the removal of $ e$. 

For a path $ P $ in $ D(b) $, we write $ \boldsymbol{A(P) }$ for the arc set and $ \boldsymbol{E(P)} $ for the vertex set of $ P $
(the latter is an edge set in the matroids). The escorting circuits of $ P $ are the escorting 
circuits of the arcs in $ A(P) $.  If $ P $ is a finite directed path 
without shortcuts  in $ D(b) $ (a shortcut  is an arc $ef\in D(b) $, where $ e,f \in E(P) $, $ f $ is later on $ P $ than $ f $ but $ ef\notin A(P) $), then 
executing simultaneously all the base exchanges corresponding to $ A(P) $ leads to a pair of bases. More precisely:
\begin{claim}\label{clm: base exchange path}
Let $ P $ be a finite directed path without shortcuts in $ D(b) $. Then the set  $ B_M' 
$  obtained from $ B_M $ by adding  the tails and deleting the heads of the arcs in $ A(P)\cap D(B_M) $ is a base of $ M $. Furthermore, $ B_M' 
$ can be obtained from $ B_M $ by a finite sequence of base exchanges where the fundamental circuits corresponding to the 
exchanges are exactly the escorting circuits of the arcs in $ A(P)\cap D(B_M) $. 

The set $ B_N' $ defined analogously (applying $ D(B_N) $ instead of $ D(B_M) $) and satisfies the analogous statement.
\end{claim}
\begin{proof}
The proof  is essentially the same as for the ``classical'' augmenting paths (see for example \cite[Lemma 3.1]{joo2021MIC}) so let 
us only give a proof sketch. Consider the 
arcs in $ 
A(P)\cap D(B_M) $ and order them according to the reverse of the path-order of $ P $. By executing the indicated base 
exchanges in $ M$ starting with $ B_M $
one by one in this order, the fundamental circuits corresponding to the later exchanges always remain the same (because there 
is no shortcut for $ P $ in $ D(b) $) and hence they can 
be executed. The argument for $ B_N $ is similar, except that we apply $ D^{-1}(B_N) $ instead of $ D(B_M) $ and the path-order of $ P $ 
instead of the  reverse path-order.
\end{proof}

We write $ \boldsymbol{b\circ P} $ for the new pair of bases 
defined in the claim 
above.  By having two disjoint $ b $-paths  without shortcuts, $ P $ and $ Q $ say, it 
may happen that executing simultaneously the corresponding  exchanges does not lead to a pair of bases. The following claim leads us to a useful 
sufficient 
condition to avoid this issue. Namely,  we will see that if $ Q 
$ does 
not 
meet any of the escorting circuits of $ P 
$, then simultaneous execution is possible. 
\begin{claim}\label{clm: no aux change}
Let $ P $ be a finite directed path without shortcuts in $ D(b) $, let $ b\circ P=:(B_M',B_N') $, and let $ U $ be the union of the escorting circuits 
of 
$ P $. Then the subgraphs of the auxiliary digraphs induced by $ E\setminus U $ are unchanged, i.e.
 $ D(B_M')[E\setminus U]=D(B_M)[E\setminus U] $ and $ D(B_N')[E\setminus U]=D(B_N)[E\setminus U] $.
\end{claim}
\begin{proof}
Suppose that $ ef \in D(B_M)[E\setminus U] $. Then $ f\in C_M(e,B_M) $ by definition. By executing the base exchanges 
corresponding to 
the arcs in $ A(P)\cap D(B_M) $ in the reversed path-order and keeping track of the changes of the fundamental circuit of $ e $  via  Corollary 
\ref{cor: fundamental change}, we conclude that every edge in $ C_M(e,B_M)\setminus U $ remains in the fundamental circuit of $ e $ after each 
exchange. Thus in 
particular $ f \in C_M(e, B_M')  $ and hence $ ef\in D(B_M') $. On the other hand, if $ e\in E\setminus (B_M \cup U) $ and $ f\in B_M 
\setminus U $ with $ ef \notin D(B_M) $, then Corollary 
\ref{cor: fundamental change} ensures that all the new edges that the fundamental circuit of $ e $ gains are coming from $ U $, i.e.
 $ C_M(e,B_M')\setminus C_M(e,B_M) \subseteq 
U  $. Thus in particular $ f\notin C_M(e,B_M') $ and hence $ ef\notin D(B_M') $. 
This concludes the proof of $ D(B_M)[E\setminus U]=D(B_M')[E\setminus U] $. The proof of $ D(B_N)[E\setminus 
U]=D(B_N')[E\setminus U] $ is similar.
\end{proof}

Changing along transfinitely many paths under similar conditions is also possible:
\begin{lemma}\label{lem: base exchange path sequence}
Let $ \mathcal{P}=\left\langle P_\alpha:\ \alpha<\xi  \right\rangle $ be a transfinite sequence of paths without shortcuts in $ D(b) $  such that for 
every $ \beta<\alpha<\xi $, $ P_\alpha $ is disjoint from all the escorting circuits of $ P_\beta $. Let $ 
A(\mathcal{P}):=\bigcup_{\alpha<\xi}A(P_\alpha) $. Then $ \boldsymbol{b \circ \mathcal{P}}:=(B_M^{\xi},B_N^{\xi})\in 
\mathcal{B}(M)\times 
\mathcal{B}(N)  $, where  $ B_M^{\xi} 
$ and   $ B_N^{\xi}$ are obtained from $ B_M $ and $ B_N $ by adding  the tails and deleting the heads of the arcs in $ A(\mathcal{P})\cap 
D(B_M) $  and in $ A(\mathcal{P})\cap D(B_N) $ respectively.

Furthermore, for the union $ U_{<\xi} $ of the escorting circuits of the paths $ P_\alpha $ for $ \alpha<\xi $,  we have $ 
D(B_M^{\xi})[E\setminus 
U_{<\xi}]=D(B_M)[E\setminus U_{<\xi}] $ and $ D(B_N^{\xi})[E\setminus U_{<\xi}]=D(B_N)[E\setminus U_{<\xi}] 
$.
\end{lemma}
\begin{proof}
For $ \alpha<\xi $, let $ \boldsymbol{U_\alpha} $ be the union of the escorting circuits of $ P_\alpha $ (w.r.t. $ b $) and let $ 
\boldsymbol{U_{<\alpha}}:=\bigcup_{\beta<\alpha}U_\beta $. We may assume 
by induction that for every $ \alpha<\xi $, we already know the 
statement for the initial segment $ 
\boldsymbol{\mathcal{P}_{<\alpha}}:=\left\langle P_\beta:\ \beta<\alpha  \right\rangle $. This means that $ \boldsymbol{b_\alpha}:=b\circ 
\mathcal{P}_{<\alpha}=: \boldsymbol{(B_M^{\alpha}, B_N^{\alpha})}\in \mathcal{B}(M)\times \mathcal{B}(N)  $, furthermore, $ 
D(B_M^{\alpha})[E\setminus 
U_{<\alpha}]=D(B_M)[E\setminus U_{<\alpha}] $ and $ D(B_N^{\alpha})[E\setminus U_{<\alpha}]=D(B_N)[E\setminus 
U_{<\alpha}] $. 

If $ \xi=0 $, then there is nothing to prove. If $ \xi $ is a successor ordinal, $ \xi=\alpha+1 $ say, then $ P_\alpha $ is a path in  $ D(b_\alpha) $ 
without shortcuts because it is such a path in $ D(b) $  and $ E(P_\alpha) \subseteq E\setminus U_{<\alpha} $ ensures via the 
induction hypothesis that this remains true in $ D(b_\alpha) $. Hence Claim \ref{clm: base exchange 
path} and the induction hypothesis ensure that $b_{\alpha+1}= b_\alpha \circ P_\alpha \in \mathcal{B}(M)\times \mathcal{B}(N) $. In the 
induction step we also need to show that  the ``furthermore'' part of the lemma is maintained.
Let $ \boldsymbol{(B_M^{\alpha+1}, B_N^{\alpha+1})}:=b_{\alpha+1} $ and let $ U^{\alpha} $  be the union of the escorting circuits of $ 
P_\alpha $ w.r.t. $ 
b_\alpha $.  Claim 
\ref{clm: no aux change}  guarantees that  $ 
D(B_M^{\alpha+1})[E\setminus U^{\alpha}]=D(B_M^{\alpha})[E\setminus U^{\alpha}] $ and $ 
D(B_N^{\alpha+1})[E\setminus U^{\alpha}]=D(B_N^{\alpha})[E\setminus U^{\alpha}] $. We know by induction that $ 
D(B_M^{\alpha})[E\setminus 
U_{<\alpha}]=D(B_M)[E\setminus U_{<\alpha}] $ and $ D(B_N^{\alpha})[E\setminus U_{<\alpha}]=D(B_N)[E\setminus 
U_{<\alpha}] $. Thus we conclude that $ 
D(B_M^{\alpha})[E\setminus 
(U_{<\alpha}\cup U^{\alpha})]=D(B_M)[E\setminus (U_{<\alpha}\cup U^{\alpha})] $ and $ D(B_N^{\alpha})[E\setminus 
(U_{<\alpha}\cup U^{\alpha})]=D(B_N)[E\setminus 
(U_{<\alpha}\cup U^{\alpha})] $. Since $ B_M^{\alpha} $ is obtained from $ B_M $ by a sequence of base exchanges,  Corollary \ref{cor: 
fundamental change}  guarantees that $ C_M(e, B_M^{\alpha}) \setminus 
C_M(e, B_M) \subseteq U_{<\alpha} $ for every $ e \in 
E(P_\alpha)\setminus B_M $. Similarly, $ C_N(e, B_N^{\alpha}) \setminus C_N(e, B_N) \subseteq U_{<\alpha} $ for every $ e\in 
E(P_\alpha)\setminus B_N $. Thus  $ 
U^{\alpha}\subseteq 
U_{<\alpha+1} $ and therefore $ 
D(B_M^{\alpha})[E\setminus 
U_{<\alpha+1}]=D(B_M)[E\setminus U_{<\alpha+1}] $ and $ D(B_N^{\alpha})[E\setminus 
U_{<\alpha+1}]=D(B_N)[E\setminus 
U_{<\alpha+1}] $.

Assume now that $ \xi $ is a limit ordinal. Then $ (B_M^{\xi},B_N^{\xi})\in \mathcal{I}(M)\times \mathcal{I}(N) $ because a circuit cannot 
appear at a limit step first. We need to show that the sets  $ B_M^{\xi} $ and $ B_N^{\xi} $ are  spanning in the corresponding matroids. Let 
$ e \in B_M\setminus B_M^{\xi} $ be given. Then there is a unique $ \alpha<\xi $ with $e\in  B_M^{\alpha}\setminus B_M^{\alpha+1} $. We 
claim that $ C_M(e,B_{\alpha+1}) $ will not change any more, i.e. for every $\beta > \alpha+1 $ we have $C_M(e,B_{\beta})= 
C_M(e,B_{\alpha+1}) $. Indeed,  for $ \beta > \alpha+1 $ the path $ P_\beta $ does not meet $ 
U_{<\alpha+1} $ by assumption and $C_M(e,B_{\alpha+1})\subseteq U^{\alpha}\subseteq U_{<\alpha+1} $. But then  $ C_M(e,B_{\alpha+1}) 
$  witnesses $ e\in 
\mathsf{span}_M(B_M^{\xi}) $. We conclude that $ 
B_M^{\xi} $ spans $ B_M $ in $ M $ and thus $ B_M^{\xi}  $  spans $ M $. The proof that $ B_N^{\xi} $ spans $ N $ is 
similar. We need to 
check again that the ``furthermore'' part of the lemma is maintained in this induction step as well. Since we already know that $ B_M^{\xi} $ and $ 
B_N^{\xi} $ span the corresponding matroids, we conclude that there cannot be any edge $ e $ where the generator $ G_M(e, B_M^{\alpha}) $ or 
$ G_N(e, 
B_N^{\alpha}) $ changes infinitely often as $ \alpha $ increases. But then for every finite  $ F\subseteq E\setminus U_{<\xi} $, there is an $ 
\alpha<\xi $ such that $ D(b_\beta)[F]=D(b_\xi)[F] $ whenever $ \alpha \leq \beta<\xi $. Therefore the ``furthermore'' part follows directly from 
the induction hypothesis.
\end{proof}

\section{Proof of Theorem \ref{thm: main}}\label{sec: main proof}

\subsection{Preprocessing}We can assume without loss of generality that $ I $ is maximal in $ \mathcal{I}(M)\cap 
\mathcal{I}(N) $ because extending 
it to a maximal element preserves the condition $ r(M/I)<r(N/I) $.
 Suppose first that $ r(N/I)\leq \aleph_0 $. Then  $ r(M/I) < \aleph_0 $ by assumption.
We apply Lemma \ref{lem: no aug hindrance}  iteratively starting with $ I $ until the second possibility occurs 
(this happens after at most $ r(M/I) $ iterations). This provides a hindrance because the iterations maintain $ r(N/I)>r(M/I) $.

Assume now that $ r(N/I) \geq \aleph_1 $. We claim that we may assume without loss of generality that $ \kappa:=r(N/I) $ is 
regular. Suppose that it is not.  Pick a regular cardinal $ \lambda $ with $ \kappa>\lambda > r(M/I) $, a base $ B $ of $ N/I $, 
and a $ \lambda $-sized  $ B' \subseteq B $.  Set $ N':=N/(B\setminus B') $ and $M':= M\setminus  (B\setminus B')$. 
Then $ r(N'/I)=\lambda >r(M/I) \geq r(M'/I) 
$ and any $ 
(M',N') 
$-hindrance is also an $ (M,N) $-hindrance by Observation \ref{obs:  hindrance minor}.    

\subsection{The popularity matroid} Let $ \boldsymbol{J_M} \in \mathcal{B}(M/I) $ and $ \boldsymbol{J_N} \in 
\mathcal{B}(N/I) $, 
furthermore,  
we denote 
$ I\cup J_M $ and $ I\cup J_N $ by $ \boldsymbol{B_M} $ and $ \boldsymbol{B_N} $ respectively. Note that $ 
J_M \cap J_N =\emptyset $ because $ I $ is assumed to be maximal in $ \mathcal{I}(M)\cap \mathcal{I}(N) $. We set $ 
\boldsymbol{b}:=(B_M,B_N) $.  We will make use of the following definitions. A 
$ b $\emph{-path} is a finite directed 
path in $ D(b) $ with $ \mathsf{in}(P)\in B_N\setminus B_M $. A $ b $\emph{-path-system} is a set $ \mathcal{P} $ of $ \kappa $ many pairwise 
disjoint $ b 
$-paths.

Informally, a finite set $ K\subseteq E $ is popular if we have a lot of opportunities to extend it to an $ M $-circuit using the edges in $ I $ and 
edges that we can add by a $ b $-path-system. The precise definition reads as 
follows:
\begin{definition}[popular sets]
 A nonempty set  $ K\subseteq E $ is  \emph{popular} (w.r.t. $ b $ and $ (M,N) $) if there is a 
 $ b $-path-system $ \mathcal{P} 
 $  such that there is a $ 
 \kappa $-sized $ 
\triangle $-system $ \mathcal{D} $ of $ M $-circuits with kernel $ K $  for which $ 
I\cup \mathsf{ter}(\mathcal{P}) $ includes all 
the petals of $ \mathcal{D} $. 
\end{definition}
Let $ \boldsymbol{\mathcal{O}} $ be the set of popular sets.
\begin{observation}\label{obs: single P witness}
For  every $ \mathcal{O}'\subseteq \mathcal{O} $ with  $ \left|\mathcal{O}'\right|\leq \kappa $, there is a single $ b 
$-path-system $ \mathcal{P} 
$ that is witnessing simultaneously the popularity of every $ O\in \mathcal{O}' $. 
\end{observation}
\begin{proof}
By taking a $ b $-path-system $ \mathcal{P}_O $ witnessing the popularity of $ O\in \mathcal{O}' $, one can build ``diagonally'' from them by a 
straightforward transfinite recursion a  $ b 
$-path-system $ \mathcal{P} $ that still provides $ \kappa $ many petals for every $ \triangle $-system corresponding to the sets in $ \mathcal{O}' 
$.
\end{proof}

\begin{lemma}\label{lem: popularity matroid}
$\boldsymbol{\mathcal{F}}:= \mathcal{C}(M)\cup \mathcal{O} $ satisfies the strong circuit elimination axiom.
\end{lemma}
\begin{proof}
Let  $ F_0, F_1\in \mathcal{F} $,  $ e\in F_0 \cap F_1 $, and $ f\in F_0 \triangle F_1 $ be given. If $ F_0, F_1\in 
\mathcal{C}(M) $, then we apply 
strong circuit elimination in $ M $ to get an $ M $-circuit $ C \subseteq (F_0 \cup F_1)\setminus \{ e \} $ through $ f $. Assume now that $ F_0 
\in \mathcal{C}(M) $ and $ F_1\in \mathcal{O} $. Then by the definition of popular sets, there is a $ b $-path-system $ 
\mathcal{P} $ 
and a  $ 
\triangle$-system $ \mathcal{D} $ of $ M $-circuits with kernel $ F_1 $ where all petals of $ \mathcal{D} $ are included in $ 
I \cup \mathsf{ter}(\mathcal{P}) $. 
By trimming $ 
\mathcal{D} $ 
we may assume that no petal meets $ F_0 $.  For 
each 
$ C\in \mathcal{D} $, let 
$ C' $ be a circuit that we obtain by applying strong circuit elimination in $ M $ with $ F_0 $,  $C $, $ e $, and $ f $.  Then 
there is a   $ G \subseteq (F_0 \cup F_1)\setminus \{ e \} $ through $ f $ for which  $ \kappa $ many $ C' $ meets $ (F_0 \cup 
F_1)\setminus \{ 
e \} $ in $ G $. But 
then these 
circuits $ C' $ form a $ \triangle $-system with kernel $ G $. Therefore $ \mathcal{P} $ witnesses that $ G $ is a popular set, thus 
$ G\in \mathcal{F} $. 

Finally, assume that  $ F_0  $ and $ F_1 $ are both popular sets.  By Observation \ref{obs: single P witness}, 
there is a single $ b $-path-system $ \mathcal{P} $ that witnesses the popularity of both. Let $ \mathcal{D}_0=\{ 
C_{0,\alpha}:\ 
\alpha<\kappa \} $ and $ 
\mathcal{D}_1=\{ C_{1,\alpha}:\ \alpha<\kappa \}  $ be  $ \triangle $-systems of $ M $-circuits with all petals included in $ 
I \cup \mathsf{ter}(P) $ with respective 
kernels $ F_0 $ and $ F_1 $. Clearly, we can assume by trimming them that every petal of $ \mathcal{D}_0 $ is disjoint to every petal 
of $ \mathcal{D}_1 $.  We obtain $ C_\alpha $ by applying strong circuit elimination 
in $ M $ with $ C_{0,\alpha} $, $ C_{1,\alpha} $, $ e $, and $ f $.  There is a $ G \subseteq (C_0\cup C_1)\setminus \{ e \} 
$ through $ f $ for which $ 
\kappa $ many $ C_\alpha $ meets $ 
(F_0\cup F_1)\setminus \{ e \} $ in $ G $. Note that the circuits $ C_\alpha $ cannot meet out of $ (F_0\cup F_1)\setminus \{ e 
\} $ 
because the petals of $ \mathcal{D}_0 $ are disjoint from
the petals of $ \mathcal{D}_1 $. But then these $ \kappa $ many $ M $-circuits form a $ \triangle $-system with kernel $ G $ 
with all petals included in $ I\cup 
\mathsf{ter}(\mathcal{P}) $. Thus $ G $ is popular. 
\end{proof}
\begin{corollary}\label{cor: pop sets are scrawls}
The minimal elements of $ \mathcal{F} $ are the circuits of a matroid $ \boldsymbol{L} $ on $ E $ in which every popular 
set is a 
scrawl (i.e. union of $ L $-circuits).
\end{corollary}
\begin{proof}
It follows directly from Lemma \ref{lem: popularity matroid} via Lemma \ref{lem: scrawl}.
\end{proof}
We call $ L $ the \emph{popularity matroid} (w.r.t. $ b $ and  $ (M,N) $). Note that $ \mathcal{I}(L)\subseteq 
\mathcal{I}(M) $ because every $ 
M 
$-circuit 
includes an $ L $-circuit.
\subsection{Bad edges}

 Let $ \boldsymbol{A_0}:=J_M\cup \mathsf{out}_{D(B_N)}(J_M) $. Suppose that the set  $ A_{n} $ is already 
 defined. Then 
  $ \boldsymbol{A_{n+1}} $ consists of: the 
 edges in $ A_n $,  those edges  $ e\in 
E\setminus B_M $ for which $ \mathsf{out}_{D(B_M)}(e)\setminus A_n $ does not $ L $-span $ e $, and  the out-neighbours of all of these edges 
in $ 
D(B_N) $. Informally, the idea behind this recursion is the following: If the $ M $-spanning of an $ e\in E\setminus 
B_M $ by $ B_M $ relies 
on 
edges that are already ``compromised'' (i.e. are in $ A_n $), then we are not confident that we will be able to $ M $-span $ e $ except if the 
uncompromised part of its generator on $ B_M $ still spans $ e $ in $ L $. We intend to get rid of the edges in $ A_n $  
(deletion in $ M 
$ and 
contraction in 
$ N $) but 
the 
price for this is compromising further edges. 

We 
set $ \boldsymbol{A}:=\bigcup_{n<\omega}A_n $ and call it the set of the \emph{bad} edges.  
\begin{observation}\label{obs: good are spanned}
For every $e\in E\setminus (B_M\cup A) 
$, the set $\mathsf{out}_M(e)\setminus A= C_M(e,B_M)\setminus (A\cup \{ e \}) $ spans $ e $ in $ L $.
\end{observation}
\begin{proof}
If $ C_M(e,B_M)\setminus (A\cup \{ e \}) $ does not span $ e $ in $ L $, then there is a smallest $ n $ such that $ C_M(e,B_M)\setminus 
(A_n\cup 
\{ e \}) $ (which is exactly $ \mathsf{out}_{D(B_M)}(e)\setminus A_n $) does not span $ e $ in $ L $ in which case $ e \in A_{n+1} $ by 
definition.
\end{proof}
\begin{observation}\label{obs: bad are N-spanned}
$ A\cap B_N $ spans $ A $ in $ N $.
\end{observation}
\begin{proof}
For each $ e\in A\setminus B_N $ the construction guarantees that $ \mathsf{out}_{D(B_N)}(e)\subseteq A $, i.e. $ 
C_N(e,B_N)\subseteq A $. 
\end{proof}

\begin{lemma}\label{lem: no kappa path unopular}
There is no $ b $-path-system $ \mathcal{P} $ with $ \mathsf{ter}(\mathcal{P}) \subseteq A  $.
\end{lemma}
\begin{proof}
Since $ \kappa $ is an uncountable regular cardinal,   it is enough to show that for every $ 
n<\omega $  there is no $ b $-path-system $ \mathcal{P} $ with  $ \mathsf{ter}(\mathcal{P}) \subseteq 
A_n  $. We apply induction on $ n $. Since $ \left|J_M\right|<\kappa $, and each $ e\in J_M $ has only finitely many out-neighbours in $ 
D(B_N) $, we conclude that $ \left|A_0\right|<\kappa $.  Assume that we know the statement  for an $ n< \omega$ and suppose for a contradiction 
that $ \mathcal{P} $ is a $ b $-path-system with $ \mathsf{ter}(\mathcal{P})\subseteq A_{n+1} $. By the induction hypotheses less than $ 
\kappa $ many paths in $ \mathcal{P} $ meets $ A_n $. Hence by trimming $ \mathcal{P} $ we can assume that no $ P\in \mathcal{P} $ meets $ 
A_n 
$. 
\begin{claim}
We can assume without loss of generality that $ \mathsf{ter}(\mathcal{P})\subseteq A_{n+1}\setminus B_M $.
\end{claim}

 \begin{proof}
 If $ \kappa $ many $ P\in \mathcal{P} 
$ meets $ A_{n+1}\setminus B_M $, then we are done by trimming $ \mathcal{P} $. Otherwise we may assume that no $ P\in \mathcal{P} $ 
meets $ 
A_{n+1}\setminus 
B_M $. Note that every $e\in (A_{n+1}\setminus A_0)\cap 
B_M  $ is an 
element of $ B_N $ (because $ B_M\setminus B_N \subseteq A_0 $) and has an in-neighbour in $ D(B_N) $  that is in $ 
A_{n+1}\setminus B_M $ 
by 
construction. Therefore every $ P\in \mathcal{P} $ can 
be 
extended by a new arc from $ D^{-1}(B_N) $ to reach $ A_{n+1}\setminus B_M $. Although these extended paths are  not 
necessarily pairwise 
disjoint, 
one terminal point is 
shared  only by finitely many of them because each  $ e\in  E\setminus  B_N$ has finite in-degree in $ D^{-1}(B_N) $. Thus we can pick $ \kappa 
$ 
many 
disjoint from them and the resulting path-system  terminates in $ A_{n+1}\setminus B_M $.
 \end{proof}

The definition of $ A_{n+1} $ ensures that for each $ e\in \mathsf{ter}(\mathcal{P})\subseteq A_{n+1}\setminus B_M $ we have $ 
\mathsf{out}_{D(B_M)}(e)\cap A_n\neq \emptyset $. By the induction hypotheses we know that it is impossible to pick pairwise distinct  
out-neighbours $ f_e\in A_n $ for $ e\in 
\mathsf{ter}(\mathcal{P}) $ in $ D(B_M) $. But then, since every $ e \in \mathsf{ter}(\mathcal{P}) $ has finite out-degree in $ D(B_M) $, the set 
$B':=\mathsf{out}_{D(B_M)}(\mathsf{ter}(\mathcal{P}))\cap A_n $ has size smaller than $ \kappa $. Thus Lemma \ref{lem: key} applied with $ 
M $, 
$ H=\mathsf{ter}(P) $, $ 
B=B_M $, and $ B'$  provides an $ e^{*}\in 
\mathsf{ter}(P)  $ and a popular set $ K $  for which $ e^{*}\in K \subseteq  C_M(e^{*},B)\setminus A_n $. Since $ K $ is a scrawl in $ 
L $ (see Corollary \ref{cor: pop sets are scrawls}), 
we conclude that  $  e^{*}\in \mathsf{span}_L(\mathsf{out}_{D(B_M)}(e^{*})\setminus A_n) $. This contradicts $e^{*}\in 
A_{n+1}\setminus B_M  $ by the definition of $ A_{n+1} $.
\end{proof}

\begin{lemma}\label{no kappa petal in A}
If $ K $ is a popular set and the $ b $-path-system $ \mathcal{P} $ and $ \triangle $-system $ \mathcal{D} $ witnessing this, then less than $ 
\kappa $  petals of $ \mathcal{D} $ meet $ A $.
\end{lemma}
\begin{proof}
There are less than $ \kappa $ petals $ p $ of $ \mathcal{D} $ with $ (p\setminus I)\cap A\neq \emptyset $ because each $ e \in p\setminus I $ is a 
terminal point of a $ P\in \mathcal{P} $ and Lemma \ref{lem: no kappa path unopular} ensures that less than $ \kappa $ many $ P\in 
\mathcal{P} $ meets $ A $. By trimming, we may assume that no $ P\in \mathcal{P} $ meets $ A $ and $ (p\setminus I)\cap A= \emptyset $ for 
each petal $ p $ of $ \mathcal{D} $. Suppose for a contradiction that there are $ \kappa $ many petals with $ p \cap I\cap A\neq \emptyset $. Then 
we can assume by trimming that $ p \cap I\cap A\neq \emptyset $ for every petal $ p $ of $ \mathcal{D} $. 
\begin{claim}\label{clm: petal arc}
For all but finitely many petals $ p $ of $ \mathcal{D} $  there 
is an arc from $ p\setminus I $ to $ p\cap I \cap A $ in $ D(B_M) $.
\end{claim}
\begin{proof}
The claim follows from Claim \ref{clm: arc from circle} combined with the fact that the kernel $ K $ of $ \mathcal{D} $ has 
a finite out-neighbourhood in 
$ D(B_M) $. In more details, we proceed as follows. We know by Claim \ref{clm: arc from circle} that for every $ C\in \mathcal{D} $ and $f\in 
C\cap B_M $ there exists an $e_{C,f}\in 
C\setminus B_M $ such that $ e_{C,f}f\in D(B_M) $.  For every $ C\in \mathcal{D} $, pick  an $ f_C\in (C\setminus 
K)\cap I\cap A $. We claim that $ e_{C,f_C}f_C $ is a 
desired 
arc for 
all but finitely many $ C\in \mathcal{D} $. Indeed, if $e_{C,f_C}\notin  p\setminus B_M=p\setminus I   $, then $ e_{C,f_C} $ must be in $ 
K\setminus B_M $. But since $ K $ is finite and every $ e\in E $ has finite out-degree in  $ D(B_M) $, this can happen only for 
finitely many $ C\in \mathcal{D} $.
\end{proof}

Each arc described in Claim \ref{clm: petal arc}  has its tail in $ \mathsf{ter}(\mathcal{P}) $ and its head in $ A $. They provide 
 extensions for $ \kappa $ many $ P\in \mathcal{P} $ to reach $ A $. Since the extended 
paths are still pairwise disjoint and terminate in $ A $, this contradicts Lemma \ref{lem: no kappa path unopular}.

\end{proof}
\subsection{Construction of a hindrance}
 Let $ \boldsymbol{\mathcal{P}}:=\left\langle P_\alpha:\ \alpha<\xi \right\rangle  $ be a maximal sequence (meaning it cannot be extended 
 further) for 
which:

\begin{enumerate}
\item\label{item: basic} $ P_\alpha $ is a $ b $-path without shortcuts for which $ \mathsf{ter}(\mathcal{P}_\alpha)\in 
E\setminus B_M $;
\item\label{item: independent} $ (B_M\setminus A)\cup \mathsf{ter}(\mathcal{P}) \in \mathcal{I}(M) $, where 
$\mathsf{ter}(\mathcal{P}):= \{ \mathsf{ter}(P_\alpha):\ \alpha<\xi \} $; 
\item\label{item: avoids} $ P_\alpha $ meets neither $ A $ nor any escorting circuit of $ P_\beta $ nor $ 
C_M(\mathsf{ter}(P_\beta), B_M)  $  for $ \beta<\alpha $;
\item\label{item: legit} for every $ \alpha<\xi $ and  $ef\in A(P_\alpha)\cap D(B_M) $, the edge $ e $ is $ M $-spanned by $ 
(B_M\setminus A)\cup \{ 
\mathsf{ter}(P_\beta):\ \beta<\alpha \} $.
\end{enumerate}
Properties (\ref{item: basic}) and (\ref{item: avoids}) ensure in particular that $b\circ \mathcal{P}=: 
(\boldsymbol{B_M'},\boldsymbol{B_N'}) $ is well defined (see Lemma \ref{lem: base exchange path sequence}). 
\begin{claim}\label{clm: xi smallal kappa}
$ \xi<\kappa $.
\end{claim}
\begin{proof}
For every $ \alpha<\xi $,  the edge $ \mathsf{ter}(P_\alpha)\in E\setminus B_M $ is not $ M $-spanned by $ B_M\setminus A $ (see properties  
(\ref{item: basic}) and (\ref{item: independent})). Therefore we must have $ 
C_M(\mathsf{ter}(P_\alpha),B_M)\cap A\neq \emptyset $.  Assume for the sake of contradiction that $ \xi \geq \kappa $. Then  there are no 
pairwise distinct $ f_\alpha \in C_M(\mathsf{ter}(P_\alpha),B_M) \cap A $ for $ \alpha<\xi $ because otherwise the corresponding extensions of 
the paths $ P_\alpha $ 
contradict Lemma \ref{lem: no kappa path unopular}. Since each $ \mathsf{ter}(P_\alpha) $ has finite out-degree in $ D(B_M) $, this implies that 
the size of the set  $\boldsymbol{G}:=\bigcup_{\alpha<\xi}C_M(\mathsf{ter}(P_\alpha),B_M) \cap A$ must be less than $ \kappa $.  The set $\{ 
\mathsf{ter}(P_\alpha):\ \alpha<\xi \}$ is $ 
M/(B_M\setminus A) $-independent  by property (\ref{item: independent}). Moreover, from the definition of $ G $ it is clear that  $\{ 
\mathsf{ter}(P_\alpha):\ \alpha<\xi \}$ is $ M/(B_M\setminus A) $-spanned by $ G $. But then in the matroid $ M/(B_M\setminus A) $, the 
independent set $\{ 
\mathsf{ter}(P_\alpha):\ \alpha<\xi \}$ of size at least $ \kappa $ is spanned by the set $ G $ where $ \left|G\right|<\kappa $, a contradiction. 
\end{proof}
\begin{claim}\label{clm: original base E minus A}
$ (B_M\setminus A)\cup \mathsf{ter}(\mathcal{P})\in \mathcal{B}(M\setminus A) $. 
\end{claim}
\begin{proof}
Let $ \boldsymbol{B}:=(B_M\setminus A)\cup \mathsf{ter}(\mathcal{P}) $. Clearly, $ B\subseteq E\setminus A $  because  no $ P_\alpha $ meets 
$ A $ (see 
property (\ref{item: avoids})). Property (\ref{item: 
independent}) ensures $ B\in \mathcal{I}(M\setminus A) $. Suppose for a contradiction, that $ e\in (E\setminus A)\setminus \mathsf{span}_M(B) 
$. 
Since $ e\in E\setminus (B_M \cup A) $, 
we know that
  $ \mathsf{out}_{D(B_M)}(e)\setminus A $ spans $ e $ in $ L $ (see Observation \ref{obs: good are spanned}).  Since in $ M $ the set $ 
  B_M\setminus A 
  \supseteq \mathsf{out}_{D(B_M)}(e)\setminus A $ does not span $ e $,  the $ L $-spanning must be through an $ L $-circuit $ \boldsymbol{K} $ 
  that is not an 
  $ M $-circuit.  Therefore there is a $ b 
  $-path-system $ 
  \boldsymbol{\mathcal{Q}} $ and a $ \kappa $-sized $ \triangle $-system $ \boldsymbol{\mathcal{D}} $ of $ M $-circuits with kernel $ K\ni e $  
  such that every 
  petal of $ 
  \mathcal{D} $ is included in $ I\cup \mathsf{ter}(\mathcal{Q}) $ and  $ K\setminus \{ e \}\subseteq  \mathsf{out}_{D(B_M)}(e)\setminus A$. 
  
By trimming $ \mathcal{Q} $ and $ \mathcal{D} $, we may assume that no $ Q\in \mathcal{Q} $ meets $ A $ (see Lemma \ref{lem: no kappa 
path unopular}) and no petal of $ \mathcal{D} $ meets $ A $ (see Lemma \ref{no kappa petal in A}). Note that $ K\setminus \{ e \}\subseteq  
\mathsf{out}_{D(B_M)}(e)\setminus 
A \subseteq B_M\setminus A \subseteq B$.  Therefore in each petal $ p $ of $ \mathcal{D} $ there exists an $ \boldsymbol{e_p}\in p $ that is not 
$ M $-spanned 
by $ B $, since otherwise for the unique $ C\in \mathcal{D} $ with $ C \supseteq p $ the set $ B $ spans each edge of $ C\setminus \{ e \} $ in $ M 
$ and thus $ e $ as well. Then $ e_p\notin I $ because $e_p \notin  B_M\setminus A=I\setminus A $ (since $ B_M\setminus B_N\subseteq A $) 
and  $ e_p \notin A $ since we trimmed $ \mathcal{D} 
$ to ensure $ p \cap A=\emptyset $. It follows that there is a $ Q_p\in \mathcal{Q} $ with $ \mathsf{ter}(Q)=e_p $. The set $ 
  \boldsymbol{\mathcal{Q}'}:=\{ Q_p:\ p \text{ is a petal of } \mathcal{D} \}\subseteq \mathcal{Q} $  has size $ \kappa $. 
  
In order to get a contradiction, we show that $ \left\langle P_\alpha:\ \alpha<\xi  \right\rangle $ can be continued.  Obviously,  $e_p\in E\setminus 
B_M $ for each petal $ p $ of $ \mathcal{D} $ because $ e_p\notin B_M\setminus A $ and $ e_p\notin A $. By shortening the paths in $ 
\mathcal{Q}' $ along shortcuts, we 
  can assume that they have no shortcuts. Since $ \xi<\kappa $ 
    (see Claim \ref{clm: xi smallal kappa}), the union $ \boldsymbol{U} $ of the escorting circuits of the paths $ P_\alpha $ and the circuit $ 
    C_M(\mathsf{ter}(P_\alpha), B_M) $ for $ \alpha<\xi $ has 
    size smaller than $ \kappa $. Pick a $ Q_p\in \mathcal{Q}' $ that avoids $ U $ and  let $ 
    P_\xi $ be the initial segment of $ Q_p $ until the first edge $ f $ that is not $ M $-spanned by $ 
    B $. Then $ f $ is well defined because $ \mathsf{ter}(Q)=e_p $ is a candidate for $ f $. Finally, the sequence $ 
    \left\langle P_\alpha:\ \alpha<\xi+1  \right\rangle $ contradicts the maximality of $ 
        \left\langle P_\alpha:\ \alpha<\xi  \right\rangle $.
\end{proof}

\begin{observation}\label{obs: same bad in BN}
$ B_N\cap A=B_N'\cap A $. 
\end{observation}
\begin{proof}
The paths $ P_\alpha $ avoid $ A $ (see property (\ref{item: avoids})), thus no $ e\in A $ is added or removed from $ B_N $ as we constructed $ 
B_N' $.
\end{proof}
We set 
$\boldsymbol{H^{*}}:= (B_M'\setminus A) \cup \mathsf{ter}(\mathcal{P})  $.
\begin{claim}\label{clm: H subset BN}
$ H^{*}\subseteq B_N' $.
\end{claim}
\begin{proof}
By property (\ref{item: basic}) each $ P_\alpha $ terminates in $ E\setminus B_M $. Thus if $ P_\alpha $ is non-trivial (i.e. 
has at least one arc), then its last arc  
must be in $ D^{-1}(B_N) $ which ensures that $ \mathsf{ter}(P_\alpha)\in B_N' $. If $ P_\alpha $ is trivial, then it consists 
of a single element $ 
e\in 
B_N\setminus B_M $ where $ e $ was never removed from $ B_N $, thus $ e\in B_N' $. This shows $ \mathsf{ter}(\mathcal{P})\subseteq B_N' $.

It remains to show that $ B_M'\setminus A \subseteq B_N' $. We have $ B_M\setminus A \subseteq B_N $, because $ B_M 
\setminus B_N \subseteq 
A $ by the definition of $ A $.  For every $ e\in B_M' 
\setminus B_M $, there exists 
an arc $ ef \in A(\mathcal{P})\cap D(B_M) $ by construction. Here either $ e\in B_N\setminus B_M $ or there exists an arc $ ge\in 
A(\mathcal{P})\cap D^{-1}(B_N) $. In 
both 
cases $ e\in B_N' $ is guaranteed. Hence $ B_M' \setminus B_M \subseteq B_N' $. Finally, if $ e\in (B_M' \cap B_M)\setminus A $, then in 
particular $ e\in B_M \setminus A\subseteq B_N $ and no $ P_\alpha $ goes through $ e $, thus $ e\in B_N' $.
\end{proof}
\begin{claim}\label{clm: N A indep}
$ H^{*}\in \mathcal{I}(N/A)$.
\end{claim}
\begin{proof}
By Claim \ref{clm: H subset BN}, $ H^{*} $ is a subset of the $ N $-base $ B_N' $. We also know that  $ A\cap H^{*}=\emptyset $ because no 
$ P_\alpha $ meets $ A $ (see property (\ref{item: avoids})). Observation \ref{obs: bad are N-spanned}
ensures $A\subseteq 
\mathsf{span}_N(B_N\cap A) $ and hence by
$ B_N\cap A=B_N'\cap A $ (Observation \ref{obs: same bad in BN}) we have $A\subseteq \mathsf{span}_N(B_N'\cap A) $. 
Thus, the union of $ H^{*}\subseteq E\setminus A $ and the base $ B_N'\cap A $ of $ N\!\!\upharpoonright\!\! A $ is $ N 
$-independent. This means $ H^{*}\in \mathcal{I}(N/A)$.
\end{proof}
\begin{claim}\label{clm: does not span N A}
$ H^{*}$ does not span $ N/A$.
\end{claim}
\begin{proof}
By Lemma \ref{lem: no kappa path unopular}, we have in particular $ \left|(B_N\setminus B_M)\cap A\right|<\kappa$ because each $ e \in 
(B_N\setminus 
B_M)\cap A$ can be considered as 
a trivial 
$ b $-path. Since $ \xi<\kappa $ (Claim \ref{clm: xi smallal kappa}) and $ \left|B_N\setminus B_M\right|=\kappa $ by the 
definition of $ \kappa $, the set 
$\boldsymbol{J}:= B_N \setminus 
(B_M \cup A\cup \{ \mathsf{\mathsf{in}_D(P_\alpha):\ \alpha<\xi} \})\subseteq B_N'\setminus B_M' $ still has size $ \kappa $. Thus in 
particular $ J\neq \emptyset $. Since $ H^{*} $, $ B_N'\cap A $, and $ J $ are pairwise disjoint subsets of $ B_N' $ and $  B_N'\cap A $ spans $ 
A 
$ in $ N $ (see Observations \ref{obs: bad are N-spanned} and \ref{obs: same bad in BN}), we conclude $ J\in \mathcal{I}(N/(A\cup H^{*})) $.
\end{proof}
\begin{lemma}\label{lem: spans M A}
$ H^{*}\in \mathcal{B}(M\setminus A) $.
\end{lemma}
\begin{proof}
By Claim \ref{clm: original base E minus A} we know that $\boldsymbol{B}:=  (B_M\setminus A)\cup \mathsf{ter}(\mathcal{P})\in 
\mathcal{B}(M\setminus A) $. Moreover, $ B_N\setminus A \in \mathcal{B}(N/A) $ because $ B_N\cap A $ spans $ A $ in $ N $ (see 
Observation \ref{obs: bad are N-spanned}).  We intend to apply Lemma \ref{lem: base exchange path sequence} with the matroid pair $ 
(M\setminus A, N/A) $, the base pair $ (B, B_N\setminus A) $ and $ \mathcal{P} $. To do so, we need to check that $ 
\mathcal{P}=\left\langle P_\alpha:\ \alpha<\xi  \right\rangle $ is a transfinite sequence of paths without shortcuts in $\boldsymbol{D}:= 
D_{M\setminus A}(B)\cup D^{-1}_{N/A}(B_N\setminus A)  $  
such that for 
every $ \beta<\alpha<\xi $, $ P_\alpha $ is disjoint from all the escorting circuits of $ P_\beta $ w.r.t. $ (M\setminus A, N/A)$  and $(B, 
B_N\setminus A) $. Clearly, $ 
E(\mathcal{P})\subseteq E\setminus A $ because of property (\ref{item: avoids}). For $ ef \in A(\mathcal{P})\cap 
D^{-1}(B_N) $, 
we have $ 
C_{N/A}(f,B_N\setminus A)=C_N(f, B_N)\setminus A \ni e $. Thus $ ef \in D^{-1}_{N/A}(B_N\setminus A) $ and no 
shortcut arc occurs in $ D^{-1}_{N/A}(B_N\setminus A) $ for the 
paths in $ \mathcal{P} $. Let  $ ef\in 
A(\mathcal{P})\cap 
D(B_M) $. Then there is 
an $ \alpha<\xi $ such that $ 
ef\in A(P_\alpha)\cap D(B_M) $. Property (\ref{item: legit}) ensures that $ C_M(e, B) \cap \mathsf{ter}(\mathcal{P})\subseteq \{ 
\mathsf{ter}(P_\beta):\ \beta<\alpha \} $. Let us execute finitely many base exchanges iteratively starting with $ B_M $ and adding each element 
of $ 
C_M(e, B) \cap \mathsf{ter}(\mathcal{P}) $ to the base while removing an element of $ A $ in each step. Note that property (\ref{item: 
independent}) ensures that it is possible and (\ref{item: legit}) 
guarantees that the fundamental circuit of $ e $ in $ M $ w.r.t. the resulting base is the same as w.r.t. $ B $. By keeping track of the changes of 
the fundamental circuit of $ e $ via Corollary \ref{cor: fundamental change} and knowing  that $ P_\alpha $ does not meet $ 
C_M(\mathsf{ter}(P_\beta),B_M) $ for $ \beta<\alpha $ (see property (\ref{item: avoids})) we have the following conclusions: $ f\in C_M(e,B) 
$ and hence $ ef\in D_{M\setminus A}(B) $, moreover,   $ C_M(e,B)\setminus C_M(e, B_M)\subseteq \bigcup_{\beta<\alpha} 
C_M(\mathsf{ter}(P_\beta))  $. Since $ \bigcup_{\beta<\alpha} 
C_M(\mathsf{ter}(P_\beta)) $ does not meet $ P_\alpha $ by property (\ref{item: avoids}), there is no shortcut arc in $ 
D_{M\setminus A}(B)  $ 
for $ P_\alpha $. Thus the paths of $ \mathcal{P} $ are paths without shortcuts in $ D $ as well.  Property (\ref{item: avoids}) together with the 
facts $ C_{N/A}(f,B_N\setminus A)=C_N(f, B_N) $ for $ f\in (E\setminus A)\setminus B_N $  and $ C_M(e,B)\setminus C_M(e, B_M)\subseteq 
\bigcup_{\beta<\alpha} 
C_M(\mathsf{ter}(P_\beta))  $ for $ e \in E(P_\alpha)\setminus B_M $  ensures that
 $ P_\alpha $ is disjoint from all 
the escorting circuits of $ P_\beta $ w.r.t. $ (M\setminus A, N/A)$  and $(B, B_N\setminus A) $.

On the one hand, Lemma \ref{lem: base exchange path sequence} ensures that the first coordinate of $ (B, B_N\setminus A)\circ \mathcal{P} $ is 
a base of $ M\setminus A $. On the other hand, it is clear from the definition of $ B $ and $ H^{*} $ that the first coordinate of
$ (B, B_N\setminus A)\circ \mathcal{P} $ is $ H^{*} $. Therefore $ H^{*}\in \mathcal{B}(M \setminus A) $.

\end{proof}

Claims \ref{clm: N A indep} and \ref{clm: does not span N A} together with Lemma \ref{lem: spans M A} imply that $ H^{*} $ is an $ 
(M\setminus A, N/A) $-hindrance. But then by Observation \ref{obs:  hindrance minor}, $ H^{*} $ is an $ 
(M,N) $-hindrance as well. Thus $ (M,N) $ is hindered which concludes the proof of Theorem \ref{thm: main}.

\printbibliography
\end{document}